\documentclass[11pt]{amsart}
\usepackage{amsmath,amsthm,amsfonts,amssymb,mathtools}
\usepackage{enumerate}
\usepackage[colorlinks=true, linkcolor=blue, citecolor=magenta, menucolor=black]{hyperref}
\usepackage{graphicx}
\usepackage[up]{caption}
\usepackage{subcaption}
\usepackage{pgf,tikz}
\usepackage[toc,page]{appendix}
\usetikzlibrary{arrows}
\usetikzlibrary{patterns}
\usepackage{color}
\usepackage{comment}
\numberwithin{equation}{section}
\newcommand{\R}{\mathbb{R}}
\newcommand{\C}{\mathbb{C}}
\newcommand{\Z}{\mathbb{Z}}
\newcommand{\Q}{\mathbb{Q}}

\newcommand{\Aut}{\operatorname{Aut}}

\newcommand{\Prob}{\operatorname{Prob}}

\newcommand{\PSL}{\operatorname{PSL}}
\newcommand{\SL}{\operatorname{SL}}

\newcommand{\PGL}{\operatorname{PGL}}

\newcommand{\rL}{\operatorname{L}}

\newcommand{\rB}{\operatorname{B}}
\newcommand{\rC}{\operatorname{C}}
\newcommand{\rE}{\operatorname{E}}

\newcommand{\rk}{\operatorname{rk}}
\newcommand{\Fix}{\operatorname{Fix}}

\newcommand{\Ad}{\operatorname{Ad}}
\newcommand{\supp}{\operatorname{supp}}

\newtheorem{theorem}{Theorem}[section]
\newtheorem{maintheorem}{Theorem}

\newtheorem{maincorollary}{Corollary}

\newtheorem{mainobstruction}{Theorem}

\newtheorem{proposition}[theorem]{Proposition}
\newtheorem{lemma}[theorem]{Lemma}
\newtheorem{corollary}[theorem]{Corollary}

\theoremstyle{definition}

\newtheorem{remark}[theorem]{Remark}

\newtheorem{example}[theorem]{Example}

\begin{document}
	
	\title[The noncommutative topological factor theorem]{The noncommutative topological factor theorem for rank-one product lattices}
	
	\author{Cyril Houdayer}
	\author{Corentin Le Bars}
	
	\address{\'Ecole normale sup\'erieure \\ D\'epartement de math\'ematiques et applications \\ Universit\'e Paris-Saclay \\ 45 rue d'Ulm \\ 75230 Paris Cedex 05 \\ France}
	\email{cyril.houdayer@ens.psl.eu}
	
	\address{\'Ecole normale sup\'erieure \\ D\'epartement de math\'ematiques et applications \\ 45 rue d'Ulm \\ 75230 Paris Cedex 05 \\ France}
	\email{corentin.le.bars@math.ens.psl.eu}
	
	\thanks{Research supported by ERC Advanced Grant NET 101141693.}

	\subjclass{46L55, 22D25, 22E40, 37A55}
	\renewcommand{\subjclassname}{\textup{2020} Mathematics Subject Classification}
	
	\keywords{Intermediate \(\rC^*\)-subalgebras, reduced crossed products, Furstenberg boundaries, irreducible lattices, rank-one groups, tree boundaries, Dani's factor theorem, ITAP}

	\begin{abstract}
		We prove a noncommutative topological factor theorem for irreducible lattices in products of real rank-one simple Lie groups. The intermediate \(\rC^*\)-subalgebras between the reduced group \(\rC^*\)-algebra and the boundary crossed product are exactly the crossed products arising from coordinate subproducts of the Furstenberg boundary. This follows from a more general theorem for product boundary actions, which also yields tree and mixed local-field versions. We finally show that, for lattices in connected semisimple real Lie groups with finite center and no compact factors, the scalar-expectation case of the corresponding classification is equivalent to ordinary {\rm ITAP}.
	\end{abstract}
	
	\maketitle
	
	\section{Introduction}
	
	A major achievement in the theory of discrete subgroups of semisimple Lie
	groups is Margulis's measurable factor theorem. Let $\Gamma<G$ be an
	irreducible lattice in a higher-rank semisimple Lie group with finite center and
	no compact factors, and let $P<G$ be a minimal parabolic subgroup. Margulis's
	theorem asserts that every measurable $\Gamma$-equivariant factor map from
	$G/P$ is, up to isomorphism, a canonical projection $G/P\to G/Q$ for an
	intermediate parabolic subgroup $P\subset Q\subset G$
	\cite[Theorem~IV.2.11(a)]{Ma91}.

	The factor theorem is a key ingredient in Margulis's celebrated normal subgroup
	theorem. Applied to a noncentral normal subgroup, it shows that the corresponding
	quotient is amenable. The complementary part of Margulis's argument shows that
	the same quotient has property~{\rm (T)} and is therefore finite. Since the
	center of the ambient group is finite, every normal subgroup of the lattice is
	either finite or of finite index \cite[Theorem~IV.4.9]{Ma91}.

	The Boutonnet--Houdayer noncommutative factor theorem gives a von Neumann
	algebraic version of Margulis's factor theorem
	\cite[Theorem~B and Theorem~3.3]{BH23}. It shows that every
	von Neumann algebra lying between the group von Neumann algebra of the
	quotient of the lattice by its center and the group measure space von Neumann
	algebra of its Furstenberg boundary comes from a unique intermediate parabolic
	quotient. More precisely, in the algebraic setting of \cite{BH23}, the base-two
	logarithm of the number of intermediate von Neumann algebras is the sum
	\(\sum_i\rk_{k_i}(\mathbf G_i)\) of the local ranks of the ambient algebraic
	factors. Thus this sum is an invariant of
	the inclusion. For connections between this theorem and Connes'\! rigidity
	conjecture for group von Neumann algebras of higher-rank lattices, we refer to
	the first author's ICM survey
	\cite[Section~5]{Ho23}.

	Dani obtained the corresponding factor theorem in the topological category. He
	proved that every continuous $\Gamma$-equivariant factor map from $G/P$ onto a
	compact Hausdorff $\Gamma$-space is, up to equivariant homeomorphism, one of the
	same parabolic projections $G/P\to G/Q$ \cite[Theorem~A]{Da84}. Thus, in both
	Margulis's and Dani's theorems, equivariance under the irreducible lattice alone
	forces the factor to come from the ambient group. For products of rank-one
	groups, these factors are precisely the coordinate subproducts of the boundary.
	Motivated by Dani's theorem and the Boutonnet--Houdayer
	noncommutative factor theorem, we classify, for rank-one product lattices, the
	intermediate \(\rC^*\)-subalgebras between the reduced group \(\rC^*\)-algebra and
	the boundary crossed product.
	Our dynamical framework also applies to tree boundaries and mixed products over
	local fields.
	
	Let \(n\geq2\), put \(N=\{1,\ldots,n\}\), and let \(2^N\) denote the power set
	of \(N\). We order \(2^N\) and all intervals of \(\rC^*\)-subalgebras by inclusion.
	Let
	\(G=G_1\times\cdots\times G_n\) be a product of second countable locally
	compact groups. Suppose that \(G_i\) acts continuously on an infinite compact
	metrizable space \(X_i\). We say that \(G_i\curvearrowright X_i\) is a
	\emph{rank-one boundary action} if \(G_i\) acts transitively on the ordered pairs
	of distinct points of \(X_i\) and contains an element with north--south
	dynamics. For \(S\subset N\), write \(X_S=\prod_{i\in S}X_i\), with
	\(X_\varnothing\) a point, and let \(p_i:G\to G_i\) be the coordinate
	projection. Denote by \(K_i\) the kernel of \(G_i\curvearrowright X_i\).

	Recall that an action of a discrete group \(\Lambda\curvearrowright X\) is
	\emph{topologically free} if \(\Fix_X(s)\) has empty interior for every
	\(s\in\Lambda\setminus\{e\}\).
	
	For a lattice \(\Gamma<G\), put
	\(\mathcal A_S=\rC(X_S)\rtimes_r\Gamma\). Thus
	\(\mathcal A_\varnothing=\rC^*_{\lambda}(\Gamma)\) and
	\(\mathcal A_N=\rC(X_N)\rtimes_r\Gamma\). The canonical expectation is denoted
	by \(\rE:\mathcal A_N\to\rC(X_N)\), and \(\rE(au_\gamma^*)\) is the
	\(\gamma\)-th Fourier coefficient of \(a\).
	
	\begin{maintheorem}\label{thm:classification}
		Assume that every \(G_i\curvearrowright X_i\) is a rank-one boundary action and that \(\Gamma<G\) is a lattice satisfying the following conditions.
		\begin{enumerate}[\rm (i)]
			\item The image of \(\Gamma\) in \(G_i/K_i\) is dense for every \(i\in N\).
			\item The discrete group \(\Gamma\) has the Haagerup--Kraus approximation property {\rm(AP)}.
			\item The coordinate action \(\Gamma\curvearrowright X_i\) is topologically free for every \(i\in N\).
		\end{enumerate}
		Then
		\[
		2^N\longrightarrow
		\left\{\mathcal D\mid
		\rC^*_{\lambda}(\Gamma)\subset\mathcal D\subset
		\rC(X_N)\rtimes_r\Gamma\right\},
		\qquad
		S\longmapsto\rC(X_S)\rtimes_r\Gamma,
		\]
		is an order isomorphism. More precisely, if \(\mathcal D=\mathcal A_S\) is the intermediate \(\rC^*\)-subalgebra corresponding to \(S\), then
		\[
		\rC^*(\rE(\mathcal D))
		=\mathcal D\cap\rC(X_N)
		=\rC(X_S).
		\]
	\end{maintheorem}

	Theorem~\ref{thm:classification} is the dynamical mechanism underlying our noncommutative version of
	Dani's theorem. For real rank-one flag manifolds, the Bruhat decomposition and
	regular elements in maximal split tori give the required boundary dynamics.
	Irreducibility, connectedness and center-freeness give effective coordinate
	density and faithful coordinate actions. Real analyticity then gives
	topological freeness. Finally, weak amenability gives {\rm (AP)}. Thus,
	Theorem~\ref{thm:classification} yields the following corollary.
	
	\begin{maincorollary}\label{cor:dani-real-rank-one-direct}
		Let \(G=G_1\times\cdots\times G_n\), where \(n\geq2\) and every \(G_i\)
		is a connected center-free noncompact simple real Lie group
		of real rank one. Let \(P_i<G_i\) be a minimal parabolic subgroup and put
		\(X_i=G_i/P_i\). Let \(\Gamma<G\) be an irreducible lattice, meaning that its
		projection to every nonempty proper subproduct is dense. Then
		\[
		2^N\longrightarrow
		\left\{\mathcal D\mid
		\rC^*_{\lambda}(\Gamma)\subset\mathcal D\subset
		\rC(X_N)\rtimes_r\Gamma\right\},
		\qquad
		S\longmapsto\rC(X_S)\rtimes_r\Gamma,
		\]
		is an order isomorphism.
	\end{maincorollary}
	
	Corollary~\ref{cor:dani-real-rank-one-direct} is the Lie-theoretic specialization of
	Theorem~\ref{thm:classification} highlighted in the
	title. The \(\rC^*\)-subalgebras in this interval, including both endpoints, are
	indexed by the subsets of the simple factors.
	Therefore, the cardinality of the interval determines the real rank of the
	ambient group. This is the reduced crossed-product analogue of the
	rank-recovery statement in the Boutonnet--Houdayer noncommutative factor
	theorem. The corollary is proved after
	Theorem~\ref{thm:classification} in
	Section~\ref{sec:proof}.
	Section~\ref{sec:applications} gives further applications to products of trees
	and mixed products over local fields. In the mixed \(S\)-arithmetic examples
	one must pass to the effective
	adjoint groups: a common finite central kernel can create additional intermediate
	\(\rC^*\)-subalgebras. In the tree case, a fixed-point criterion of
	Bader--Boutonnet--Houdayer--Peterson makes topological freeness automatic for
	uniform lattices in products of thick biregular trees with injective coordinate
	projections \cite[Proposition~6.4]{BBHP22}. For the simple irreducible
	Burger--Mozes lattices, this gives an explicit classification consisting of four
	intermediate \(\rC^*\)-subalgebras \cite{BM00b}.
	
	\subsection*{Relation to intermediate crossed-product correspondences}
	A large part of the literature concerns the Galois problem for a crossed-product
	inclusion
	\[
	\mathcal C\subset\mathcal D\subset\mathcal C\rtimes_r\Gamma.
	\]
	Here the coefficient algebra is fixed, while the partial subaction varies. Choda
	obtained an early result under a conditional-expectation hypothesis \cite{Ch80},
	while Landstad, Olesen, and Pedersen treated abelian groups \cite{LOP78}. Izumi
	proved the subgroup correspondence for outer finite-group actions on simple
	algebras \cite{Iz02}, and Cameron and Smith extended it to outer actions of
	arbitrary discrete groups on unital simple algebras \cite{CS19}. For free actions
	of groups with {\rm(AP)} on commutative algebras, Brown, Exel, Fuller, Pitts, and
	Reznikoff obtained the correspondence with partial subactions \cite{BEFPR21}.
	Kennedy and Ursu then gave a
	nearly complete criterion for groups with {\rm(AP)}. Pointwise residual proper
	outerness is sufficient and, except for the order-two phenomenon they isolate
	in the noncommutative case, necessary for every intermediate algebra to arise
	from a partial subaction \cite{KU24}.

	Theorem~\ref{thm:classification} concerns the dual inclusion
	\[
	\rC^*_{\lambda}(\Gamma)\subset\mathcal D
	\subset\rC(X_N)\rtimes_r\Gamma.
	\]
	Every canonical group unitary is already present, and the coefficient algebra
	varies. The theorem recovers it from the diagonal as
	\[
	\mathcal D\cap\rC(X_N)=\rC(X_S).
	\]
	Thus a Galois correspondence fixes the coefficient algebra and determines a
	partial subaction, whereas our result fixes the group part and determines a
	coordinate factor of \(X_N\).

	Closer to our setting are Suzuki's correspondence between intermediate
	dynamical extensions and intermediate crossed products, Ochi's weakening of its
	freeness assumption, and Amrutam's reconstruction results under different
	group-theoretic hypotheses \cite{Su20,Oc21,Am21}. Suzuki's and Ochi's
	correspondences do not apply to the factor map \(X_N\to X_\varnothing\), which
	satisfies neither Suzuki's freeness assumption nor Ochi's singleton-fiber
	condition over the nonfree part.
	Stationary states and Powers averaging have also been used to prove simplicity
	of intermediate \(\rC^*\)-algebras \cite{AK20,AU22}. Boundary maps and
	noncommutative boundaries give results on simplicity and ideal structure
	\cite{KS22,KS19}, rather than a classification of the intermediate lattice.

	\subsection*{Idea of the proof}
	The proof combines the commutative factor theorem with Powers averaging to pass
	from Fourier coefficients to the diagonal part. Let \(\mathcal D\) be an
	intermediate algebra. The commutative product-boundary factor theorem gives subsets
	\(S,T\subset N\) such that
	\[
	\rC^*(\rE(\mathcal D))=\rC(X_S)
	\quad\text{and}\quad
	\mathcal D\cap\rC(X_N)=\rC(X_T).
	\]
	The first algebra contains all Fourier coefficients of \(\mathcal D\), but need
	not be contained in \(\mathcal D\), since the expectation need not preserve an
	intermediate algebra. Suzuki's Fourier reconstruction theorem gives
	\(\mathcal D\subset\rC(X_S)\rtimes_r\Gamma\), and hence \(T\subset S\).

	For the reverse inclusion, fix \(j\in S\) and choose \(d\in\mathcal D\) such that
	\(f=\rE(d)\) depends on the \(j\)-th coordinate. Choose \(k\neq j\).
	Topological freeness provides a point \(\xi\in X_k\) with trivial stabilizer
	such that the slice \(\sigma_{k,\xi}(f)\) still depends on the \(j\)-th
	coordinate. Rank-one north--south dynamics and Poincar\'e recurrence produce
	lattice elements that are close to the identity outside the \(k\)-th factor and
	send the complement of a neighborhood of \(\xi\) into pairwise disjoint regions
	of \(X_k\). In the corresponding conjugation averages, Powers' estimate makes
	the nonzero Fourier terms tend to zero, while the zeroth coefficient converges
	uniformly to \(\sigma_{k,\xi}(f)\). The slice therefore belongs to
	\(\mathcal D\cap\rC(X_N)=\rC(X_T)\), so \(j\in T\). Thus \(S=T\), and the
	diagonal together with the canonical unitaries gives
	\(\mathcal D=\rC(X_S)\rtimes_r\Gamma\).

	The recurrence lemma is used twice. In the commutative factor theorem, one
	coordinate is kept nearly fixed while all the others are contracted, which
	forces invariant equivalence relations to be coordinate relations. In the
	averaging argument, all but one coordinate are kept nearly fixed and the
	remaining coordinate is contracted.
	
	Beyond the approximation-property setting, even the scalar-expectation case
	presents a separate obstruction. For lattices in connected semisimple real Lie
	groups, it is controlled exactly by the invariant translation approximation
	property.
	
	\begin{mainobstruction}\label{thm:semisimple-itap-obstruction}
		Let \(G\) be a connected semisimple real Lie group with finite center and
		without compact factors, let \(\Gamma<G\) be a lattice, and let \(P<G\) be a
		minimal parabolic subgroup. Put
		\(\mathcal A=\rC(G/P)\rtimes_r\Gamma\), and let
		\(\rE:\mathcal A\to\rC(G/P)\) be the canonical expectation. Write
		\(\rC_u^*(\Gamma)=\ell^\infty(\Gamma)\rtimes_r\Gamma\) for the uniform Roe
		\(\rC^*\)-algebra and \(\rL(\Gamma)\) for the group von Neumann algebra. The following are
		equivalent.
		\begin{enumerate}[\rm (i)]
			\item \(\Gamma\) has ordinary {\rm ITAP}, that is,
			\[
			\rC_u^*(\Gamma)\cap\rL(\Gamma)=\rC^*_{\lambda}(\Gamma).
			\]
			\item Every intermediate \(\rC^*\)-subalgebra
			\(\rC^*_{\lambda}(\Gamma)\subset\mathcal D\subset\mathcal A\) with
		\(\rC^*(\rE(\mathcal D))=\C1\) is equal to
			\(\rC^*_{\lambda}(\Gamma)\).
		\end{enumerate}
		Consequently, if every intermediate \(\rC^*\)-subalgebra is of the form
		\(\rC(G/Q)\rtimes_r\Gamma\) for a parabolic subgroup
		\(P\subset Q\subset G\), then \(\Gamma\) has {\rm ITAP}.
	\end{mainobstruction}

	Although {\rm ITAP} is formulated in the generally nonseparable uniform Roe
	algebra, Proposition~\ref{prop:amenable-orbit-saturation} makes its defining
	intersection visible in a single separable crossed product. If a countable group
	\(\Lambda\) acts topologically amenably on a compact metrizable space \(X\) with
	a dense orbit \(\Lambda x_0\), then
	\[
	\pi_{x_0}\bigl(\rC(X)\rtimes_r\Lambda\bigr)\cap\rL(\Lambda)
	=
	\rC_u^*(\Lambda)\cap\rL(\Lambda).
	\]
	Every exact countable group admits such an action after passing to the closure
	of one orbit. For Theorem~\ref{thm:semisimple-itap-obstruction}, the amenable
	minimal action on \(G/P\) supplies this model, while
	Lemma~\ref{lem:itap-orbit-representation} identifies the left-hand side with
	the elements whose Fourier coefficients are scalar. These elements form the
	largest intermediate algebra with scalar expectation, so the displayed equality
	turns condition~{\rm (ii)} exactly into {\rm ITAP}.
	
	The final implication is one-way: {\rm ITAP} settles exactly the scalar upper-bound case and does not by itself imply the full noncommutative Dani classification. Theorem~\ref{thm:semisimple-itap-obstruction} is proved independently of Theorem~\ref{thm:classification} in Section~\ref{sec:scalar-itap}.
	
	The main operator-algebraic input is Suzuki's slice-map theorem
	\cite[Proposition~3.4]{Su17}. For approximation properties we use
	\cite{CH89,HK94,Sz91}. The dynamical argument uses Poincar\'e recurrence and
	Powers' method \cite{Po75}.

	\subsection*{AI statement}
	The starting point of this project was a proof, for irreducible lattices in products of real rank-one Lie groups, of the upper bound that now appears in Theorem~\ref{thm:classification}. The authors' initial approach to the matching lower bound was to combine Dani's topological factor theorem with the von Neumann algebraic factor theorem of Boutonnet--Houdayer. This did not turn out to be the right route. Extensive interaction with ChatGPT~5.6 Sol led instead to the purely \(\rC^*\)-algebraic proof presented here. Its key idea is to use rank-one geometry simultaneously to recover coordinate slices of the diagonal and, through Powers averaging, to eliminate the nonzero Fourier coefficients. ChatGPT was also used to explore the scalar-expectation case and its relation to {\rm ITAP}, to explore the extension to more general products of rank-1 groups, and to draft parts of the paper. All arguments were checked and substantially revised by the authors, who take full responsibility for the final mathematical content.
	
	\tableofcontents
	
	\section{The abstract product-boundary factor theorem}
	\label{sec:abstract-coordinate-factor}
	
	In this section we prove the commutative core of Theorem~\ref{thm:classification}.  We isolate a
	dynamical criterion that gives the coordinate form of Dani's topological factor
	theorem independently of \cite{Da84}.  The criterion applies uniformly to real
	rank-one boundaries, tree boundaries, and mixed products over local fields.  Its
	recurrence input will also drive the noncommutative Powers-averaging argument.
	
	Let $n\geq2$ and let $G=G_1\times\cdots\times G_n$, where every $G_i$ is a
	second countable locally compact group acting continuously
	on an infinite compact metrizable space $X_i$.  Put
	\[
	N=\{1,\ldots,n\},\qquad
	G_I=\prod_{i\in I}G_i,\qquad
	X_I=\prod_{i\in I}X_i
	\]
	for $I\subset N$, with $G_\varnothing=\{e\}$ and $X_\varnothing=\{*\}$, and denote by
	$p_I:G\to G_I$ and $\operatorname{pr}_I:X_N\to X_I$ the coordinate projections.
	We say that $G_i\curvearrowright X_i$ is a \emph{rank-one boundary action}
	if the following conditions hold.
	\begin{enumerate}[\rm (i)]
		\item The diagonal action of $G_i$ on
		$X_i^{(2)}=\{(\xi,\eta)\in X_i\times X_i\mid \xi\neq\eta\}$ is transitive.
		\item There is an element $a_i\in G_i$ with two distinct fixed points
		$a_i^-,a_i^+\in X_i$ such that, for every pair of neighborhoods
		$O^-\ni a_i^-$ and $V^+\ni a_i^+$, there is $m_0\geq1$ such that
		$a_i^m(X_i\setminus O^-)\subset V^+$ for every $m\geq m_0$. Such an element
		is called a \emph{north--south element}.
	\end{enumerate}
	The first condition implies that $G_i$ is transitive on $X_i$ and that $X_i$
	has no isolated points.  It also implies that, after conjugating a fixed
	north--south element, one obtains a north--south element with any prescribed
	ordered pair of distinct points as its repelling and attracting fixed points.
	
	Let $\Gamma<G$ be a lattice. For every $i\in N$, put
	$K_i=\ker(G_i\curvearrowright X_i)$ and assume that the image of $\Gamma$
	in the effective quotient $G_i/K_i$ is dense. For the rest of this section, we keep these
	assumptions on the actions and on $\Gamma$. Neither {\rm(AP)} nor topological
	freeness is needed until the noncommutative part of the proof.
	
	\subsection{Recurrence and simultaneous contraction}
	
	The following proposition is the only place in the coordinate-factor argument
	where the lattice hypothesis is used.  Its codimension-one special case is also
	the recurrence statement needed in Section~\ref{sec:proof}.
	
	\begin{proposition}
		\label{lem:abstract-recurrence-contraction}
		\label{prop:codim-one-strip}
		Let $\varnothing\neq I\subsetneq N$ and put $J=N\setminus I$.  Let
		$U\subset G_I$ be an identity neighborhood.  For every $j\in J$, let
		$L_j\subsetneq X_j$ be compact and let $V_j\subset X_j$ be nonempty and open.
		Then there exists $\gamma\in\Gamma$ such that
		\begin{equation}
			\label{eq:simultaneous-recurrence-contraction}
			p_I(\gamma)\in U
			\quad\text{and}\quad
			p_j(\gamma)L_j\subset V_j
			\quad\text{for every }j\in J.
		\end{equation}
		
		In particular, if $j\in N$, $U\subset G_{N\setminus\{j\}}$ is an identity
		neighborhood, and $O,V\subset X_j$ are nonempty open sets with
		$X_j\setminus O\neq\varnothing$, then there exists $\gamma\in\Gamma$ such that
		\begin{equation}
			\label{eq:codim-one-strip}
			p_{N\setminus\{j\}}(\gamma)\in U,
			\qquad
			p_j(\gamma)(X_j\setminus O)\subset V.
		\end{equation}
	\end{proposition}
	
	\begin{proof}
		Fix $j\in J$. Choose $\xi_j^-\in X_j\setminus L_j$ and
		$\xi_j^+\in V_j\setminus\{\xi_j^-\}$.
		The second choice is possible because $X_j$ has no isolated points.  By
		conjugating a north--south element, choose $a_j\in G_j$ whose repelling and
		attracting fixed points are respectively $\xi_j^-$ and $\xi_j^+$.  Choose open
		neighborhoods $O_j\ni\xi_j^-$ and $W_j\ni\xi_j^+$ such that
		$\overline{O_j}\cap L_j=\varnothing$ and $\overline{W_j}\subset V_j$.
		By north--south dynamics, there is $m_j\geq1$ such that
		\begin{equation}
			\label{eq:abstract-ns-j}
			a_j^m(X_j\setminus O_j)\subset W_j
			\qquad\text{for every }m\geq m_j.
		\end{equation}
		Continuity of the action and compactness of $L_j$ and $\overline{W_j}$ give an
		identity neighborhood $B_j\subset G_j$ such that
		\begin{equation}
			\label{eq:abstract-small-perturbations}
			B_jL_j\subset X_j\setminus O_j,
			\qquad
			B_j^{-1}\overline{W_j}\subset V_j.
		\end{equation}
		
		Choose an identity neighborhood $B_I\subset G_I$ satisfying
		$B_I^{-1}B_I\subset U$, and put
		\[
		B=B_I\times\prod_{j\in J}B_j\subset G,
		\qquad
		a=(e_I,(a_j)_{j\in J})\in G.
		\]
		Let $\pi:G\to G/\Gamma$ be the quotient map.  The finite $G$-invariant measure
		on $G/\Gamma$ has full support, so the nonempty open set $\pi(B)$ has positive
		measure.  Left translation by $a$ is invertible and measure preserving.  By
		Poincar\'e recurrence, there are arbitrarily large integers $m$ such that
		$a^m\pi(B)\cap\pi(B)\neq\varnothing$.
		Choose such an $m$ with $m\geq\max_{j\in J}m_j$.  There are $b_1,b_2\in B$ such
		that $a^mb_1\Gamma=b_2\Gamma$. Hence
		$\gamma=b_2^{-1}a^mb_1\in\Gamma$. Writing
		$b_\ell=(b_{\ell,I},(b_{\ell,j})_{j\in J})$, we obtain
		$p_I(\gamma)=b_{2,I}^{-1}b_{1,I}\in B_I^{-1}B_I\subset U$.
		For $j\in J$, equations \eqref{eq:abstract-ns-j} and
		\eqref{eq:abstract-small-perturbations} give
		\[p_j(\gamma)L_j
			=b_{2,j}^{-1}a_j^mb_{1,j}L_j
			\subset B_j^{-1}a_j^m(X_j\setminus O_j)
			\subset B_j^{-1}W_j
			\subset V_j.\]
		This proves \eqref{eq:simultaneous-recurrence-contraction}. The final assertion
		is obtained by taking $I=N\setminus\{j\}$ and $L_j=X_j\setminus O$. Here $I$ is
		nonempty because $n\geq2$.
	\end{proof}
	
	\subsection{Closed invariant equivalence relations}
	
	For $i\in N$, write $\widehat i=N\setminus\{i\}$ and identify
	$X_N=X_{\widehat i}\times X_i$ when convenient.
	
	\begin{lemma}
		\label{lem:abstract-one-coordinate-saturation}
		Let $\mathcal R\subset X_N\times X_N$ be a closed $\Gamma$-invariant equivalence
		relation.  Fix $i\in N$.  If there exist $x,y\in X_N$ such that
		$(x,y)\in\mathcal R$ and $x_i\neq y_i$, then
		\[
		\bigl((z,\xi),(z,\eta)\bigr)\in\mathcal R
		\]
		for every $z\in X_{\widehat i}$ and every $\xi,\eta\in X_i$.
	\end{lemma}
	
	\begin{proof}
		We first produce an equivalent pair which differs only in the $i$-th
		coordinate.  Fix $z=(z_j)_{j\neq i}\in X_{\widehat i}$.  Choose a decreasing sequence of open neighborhoods $(U_m)_m$ of $e$ in $G_i$ with $\bigcap_m U_m = \{e\}$. For every $j\neq i$, choose in a similar fashion a decreasing sequence of open neighborhoods $(V_{j,m})_m$ of $z_j$ shrinking to $\{z_j\}$.  Apply
		Proposition~\ref{lem:abstract-recurrence-contraction} with $I=\{i\}$ and
		$L_j=\{x_j,y_j\}$ for $j\neq i$. We obtain $\gamma_m\in\Gamma$ such that
		$p_i(\gamma_m)\to e$, while $p_j(\gamma_m)x_j\to z_j$ and
		$p_j(\gamma_m)y_j\to z_j$ for every $j\neq i$.
		Since $\mathcal R$ is $\Gamma$-invariant and closed,
		\begin{equation}
			\label{eq:abstract-purified-pair}
			\bigl((z,x_i),(z,y_i)\bigr)\in\mathcal R.
		\end{equation}
		
		Put $\alpha=x_i$ and $\beta=y_i$, so $\alpha\neq\beta$.  Let
		$w\in X_{\widehat i}$ and let $\xi,\eta\in X_i$.  There is nothing to prove if
		$\xi=\eta$, so assume $\xi\neq\eta$. By transitivity on $X_i^{(2)}$, choose
		$g_i\in G_i$ such that $g_i\alpha=\xi$ and $g_i\beta=\eta$.
		The image of $\Gamma$ in $G_i/K_i$ is dense, so there is a sequence
		$s_m\in\Gamma$ whose images in $G_i/K_i$ converge to the image of $g_i$.
		The action factors continuously through $G_i/K_i$, so
		$p_i(s_m)\alpha\to\xi$ and $p_i(s_m)\beta\to\eta$.
		Applying $s_m$ to \eqref{eq:abstract-purified-pair} gives
		\begin{equation}
			\label{eq:abstract-moving-common-coordinate}
			\left(
			\bigl(p_{\widehat i}(s_m)z,p_i(s_m)\alpha\bigr),
			\bigl(p_{\widehat i}(s_m)z,p_i(s_m)\beta\bigr)
			\right)\in\mathcal R.
		\end{equation}
		
		Choose identity neighborhoods $U_m\subset G_i$ shrinking to $e$, and, for
		$j\neq i$, neighborhoods $W_{j,m}$ of $w_j$ shrinking to $w_j$.  Apply
		Proposition~\ref{lem:abstract-recurrence-contraction}, again with $I=\{i\}$,
		to the singleton compact sets $L_{j,m}=\{p_j(s_m)z_j\}$ for $j\neq i$. This
		gives $t_m\in\Gamma$ such that $p_i(t_m)\to e$ and
		$p_j(t_ms_m)z_j\to w_j$ for every $j\neq i$.
		Apply $t_m$ to \eqref{eq:abstract-moving-common-coordinate}.  By joint
		continuity of the actions, the resulting related pairs converge to
		$(w,\xi)$ and $(w,\eta)$, respectively.  Closedness of $\mathcal R$ gives
		$\bigl((w,\xi),(w,\eta)\bigr)\in\mathcal R$.
	\end{proof}
	
	\begin{theorem}
		\label{thm:abstract-coordinate-equivalence}
		Every closed $\Gamma$-invariant equivalence relation $\mathcal R$ on $X_N$ is a
		coordinate equivalence relation.  More precisely, there is a unique subset
		$J\subset N$ such that
		\begin{equation}
			\label{eq:abstract-coordinate-relation}
			(x,y)\in\mathcal R
			\quad\Longleftrightarrow\quad
			x_j=y_j\text{ for every }j\in N\setminus J.
		\end{equation}
	\end{theorem}
	
	\begin{proof}
		Let $J\subset N$ consist of those coordinates $i$ for which there is a related
		pair $(x,y)\in\mathcal R$ satisfying $x_i\neq y_i$.
		If $i\in J$, Lemma~\ref{lem:abstract-one-coordinate-saturation} shows that the
		$i$-th coordinate may be changed arbitrarily without leaving an $\mathcal R$-class.
		Thus, if $x,y\in X_N$ agree outside $J$, one can pass from $x$ to $y$ by
		changing the coordinates in $J$ one at a time. Transitivity of $\mathcal R$ gives
		$(x,y)\in\mathcal R$. Conversely, if $(x,y)\in\mathcal R$ and $x_i\neq y_i$, then
		$i\in J$ by definition.  Hence every related pair agrees outside $J$, which
		proves \eqref{eq:abstract-coordinate-relation}.  The same formula makes $J$
		unique.
	\end{proof}
	
	\begin{theorem}
		\label{thm:abstract-topological-factor}
		Let $Y$ be a compact Hausdorff $\Gamma$-space and let
		$q:X_N\to Y$ be a continuous surjective $\Gamma$-equivariant map.  Then there
		is a unique subset $S\subset N$ and a $\Gamma$-equivariant homeomorphism
		$h:X_S\to Y$ such that
		\begin{equation}
			\label{eq:abstract-factorization}
			q=h\circ\operatorname{pr}_S.
		\end{equation}
	\end{theorem}
	
	\begin{proof}
		Define $\mathcal R_q$ by $(x,y)\in\mathcal R_q$ when $q(x)=q(y)$. This is a closed
		$\Gamma$-invariant equivalence relation on $X_N$. By
		Theorem~\ref{thm:abstract-coordinate-equivalence}, there is $J\subset N$ such
		that two points have the same image under $q$ exactly when they agree outside
		$J$.  Put $S=N\setminus J$.  Then $q$ and $\operatorname{pr}_S$ have the same
		fibers, so there is a unique bijection $h:X_S\to Y$ satisfying
		\eqref{eq:abstract-factorization}.  Since $\operatorname{pr}_S$ is a quotient
		map, $h$ is continuous.  Since $X_S$ is compact and $Y$ is Hausdorff, $h$ is a
		homeomorphism.  Equivariance follows from that of $q$ and
		$\operatorname{pr}_S$, and uniqueness of $S$ follows from the description of
		the fiber relation.
	\end{proof}
	
	The preceding result has the following commutative formulation.  We give it a
	separate label because this is the form used throughout the crossed-product
	argument.
	
	\begin{theorem}
		\label{thm:dani-coordinate}
		Every unital $\Gamma$-invariant $\rC^*$-subalgebra
		$\mathcal B\subset\rC(X_N)$ is of the form
		\[
		\mathcal B=\rC(X_S)
		\]
		for a unique subset $S\subset N$, where functions on $X_S$ are identified with
		their pullbacks to $X_N$.
	\end{theorem}
	
	\begin{proof}
		Let $Y$ be the Gelfand spectrum of $\mathcal B$. The inclusion
		$\mathcal B\subset\rC(X_N)$ is induced by the continuous map $q:X_N\to Y$
		defined by $q(x)=\operatorname{ev}_x|_{\mathcal B}$.
		The map is $\Gamma$-equivariant.  Its image is compact and hence closed, and it
		is dense because the pullback $q^*:\rC(Y)\to\rC(X_N)$ is injective.  Hence $q$
		is surjective.  By Theorem~\ref{thm:abstract-topological-factor}, there are
		$S\subset N$ and a homeomorphism $h:X_S\to Y$ such that
		$q=h\circ\operatorname{pr}_S$.  Therefore
		\[
		\mathcal B=q^*\rC(Y)
		=\operatorname{pr}_S^*\rC(X_S)
		=\rC(X_S).
		\]
		Uniqueness follows from the uniqueness of $S$ in
		Theorem~\ref{thm:abstract-topological-factor}.
	\end{proof}
	
	\section{Fourier reconstruction and dynamical reductions}
	
	From now until the end of the proof of Theorem~\ref{thm:classification} in Section~\ref{sec:proof},
	we work under all the hypotheses of Theorem~\ref{thm:classification}.
	
	\subsection{Free points in the coordinate actions}
	
	If $S\subset T\subset N$, the coordinate projection $X_T\to X_S$ identifies
	$\rC(X_S)$ with a $\rC^*$-subalgebra of $\rC(X_T)$.  We use these identifications
	without further comment.  Since every rank-one boundary $X_i$ is nontrivial,
	the coordinate $\rC^*$-algebras $\rC(X_S)\subset\rC(X_N)$ are distinct for distinct
	$S\subset N$.
	
	\begin{lemma}
		\label{lem:coordinate-action-topologically-free}
		For every $i\in N$, the points of $X_i$ with trivial stabilizer for the
		coordinate action of $\Gamma$ form a dense $G_\delta$.
	\end{lemma}
	
	\begin{proof}
		A lattice in a second countable locally compact group is countable. For every
		$t\in\Gamma\setminus\{e\}$, topological freeness says that the closed set
		$\Fix_{X_i}(t)$ has empty interior, so its complement is open and dense. The
		Baire category theorem shows that the intersection of these complements is a
		dense $G_\delta$. This is precisely the set of points with trivial stabilizer.
	\end{proof}
	
	The recurrence statement used later is already contained in
	Proposition~\ref{prop:codim-one-strip}: its codimension-one formulation is
	\eqref{eq:codim-one-strip}.
	
	\subsection{Fourier reconstruction and the upper bound}
	
	For a $\Gamma$-$\rC^*$-algebra $\mathcal C$, we write
	$\mathcal C\rtimes_r\Gamma$ for the reduced crossed product and denote by
	$\rE:\mathcal C\rtimes_r\Gamma\to\mathcal C$ the canonical expectation,
	characterized by $\rE(au_\gamma)=\delta_{\gamma,e}a$ \cite{BO08}. We use the
	same symbol for the expectations on all coordinate
	crossed products.  The coordinate classification needed below is
	Theorem~\ref{thm:dani-coordinate}. It was proved once, in the abstract setting,
	in Section~\ref{sec:abstract-coordinate-factor}.
	
	Recall the Haagerup--Kraus approximation property (AP). By
	\cite[Theorem~1.9(c)]{HK94}, for a discrete group this is equivalent to the
	existence of a net of finitely supported completely bounded Fourier multipliers
	converging to the identity in the stable point-norm topology. See also
	\cite[Definition~2.3]{Su17}. In the abstract setting,
	{\rm(AP)} for $\Gamma$ is a hypothesis of Theorem~\ref{thm:classification}. For
	Corollary~\ref{cor:dani-real-rank-one-direct} and the
	two concrete families treated in Section~\ref{sec:applications}, this hypothesis
	is verified in the corresponding proofs.
	
	The relevance of {\rm(AP)} is the following Fubini-type reconstruction theorem
	of Suzuki \cite[Proposition~3.4]{Su17}.  No freeness assumption on the action is
	involved. The following proposition is independent of the standing
	product-boundary hypotheses.
	
	\begin{proposition}
		\label{prop:suzuki}
		Suppose $\Gamma$ has {\rm(AP)}.  Let $\Gamma\curvearrowright\mathcal C$ be an
		action on a unital $\rC^*$-algebra and let
		$\mathcal C_0\subset\mathcal C$ be a $\Gamma$-invariant unital
		$\rC^*$-subalgebra.  Then, inside $\mathcal C\rtimes_r\Gamma$,
		\[
		\mathcal C_0\rtimes_r\Gamma
		=
		\{a\in\mathcal C\rtimes_r\Gamma\mid
		\rE(au_\gamma^*)\in\mathcal C_0
		\text{ for every }\gamma\in\Gamma\}.
		\]
	\end{proposition}
	
	\begin{proof}
		Suzuki's result says that if all Fourier coefficients of $a$ belong to the
		closed subspace $\mathcal C_0$, then
		$a\in\overline{\operatorname{span}}\{cu_\gamma\mid c\in\mathcal C_0,
		\ \gamma\in\Gamma\}$.
		Since $\mathcal C_0$ is $\Gamma$-invariant, this closed linear span is precisely
		$\mathcal C_0\rtimes_r\Gamma$.  The converse follows first for finite Fourier
		polynomials and then by norm continuity of the Fourier coefficient maps.
	\end{proof}

	Suzuki also proves the converse. If this reconstruction property holds for every
	$\Gamma$-$\rC^*$-algebra and every closed subspace, then $\Gamma$ has {\rm(AP)}
	\cite[Proposition~3.4]{Su17}. Thus {\rm(AP)} is exactly the hypothesis behind this
	general Fourier-reconstruction route.
	
	Combining Theorem~\ref{thm:dani-coordinate} with
	Proposition~\ref{prop:suzuki} gives the upper bound directly at the
	$\rC^*$-level.
	
	\begin{proposition}
		\label{prop:dani-ap-upper}
		For every intermediate $\rC^*$-subalgebra $\mathcal D$ of $\mathcal A_N$
		containing $\rC^*_{\lambda}(\Gamma)$, there is a unique $S\subset N$ such that
		\[
		\rC^*(\rE(\mathcal D))=\rC(X_S).
		\]
		For this $S$, one has $\mathcal D\subset\mathcal A_S$.
	\end{proposition}
	
	\begin{proof}
		The set $\rE(\mathcal D)$ need not be closed under multiplication, so consider
		the $\rC^*$-algebra $\rC^*(\rE(\mathcal D))$ that it generates. Since
		$\mathcal D$ contains the canonical unitaries, this $\rC^*$-algebra is
		$\Gamma$-invariant. Indeed, for $d\in\mathcal D$ and
		$\gamma\in\Gamma$, one has $u_\gamma d u_\gamma^*\in\mathcal D$ and
		$\gamma\cdot\rE(d)=\rE(u_\gamma d u_\gamma^*)$.
		Theorem~\ref{thm:dani-coordinate} gives a unique $S\subset N$ such that
		$\rC^*(\rE(\mathcal D))=\rC(X_S)$.
		
		Let $d\in\mathcal D$ and $\gamma\in\Gamma$. Since $u_\gamma^*\in\mathcal D$,
		one has $du_\gamma^*\in\mathcal D$ and therefore
		$\rE(du_\gamma^*)\in\rE(\mathcal D)\subset\rC(X_S)$.
		The group $\Gamma$ has {\rm(AP)}, so Proposition~\ref{prop:suzuki}, applied to
		$\mathcal C=\rC(X_N)$ and $\mathcal C_0=\rC(X_S)$, gives
		$d\in\mathcal A_S$.  Hence $\mathcal D\subset\mathcal A_S$.
	\end{proof}

	\section{Proofs of Theorem~\ref{thm:classification} and Corollary~\ref{cor:dani-real-rank-one-direct}}\label{sec:proof}
	
	\subsection{Powers averaging}
	
	We first recall the support estimate used to remove nonzero Fourier coefficients. It is an elementary form of Powers' averaging method \cite{Po75}.
	
	\begin{lemma}\label{lem:powers-criterion}
		Let \(\Gamma\curvearrowright\mathcal C\) be an action on a unital
		\(\rC^*\)-algebra and let \(t\in\Gamma\setminus\{e\}\). Assume that there are
		a partition \(\Gamma=C\sqcup D\) and elements
		\(s_1,\ldots,s_m\in\Gamma\) such that \(tC\cap C=\varnothing\) and the subsets
		\(s_1D,\ldots,s_mD\) are pairwise disjoint. If
		\(b_1,\ldots,b_m\in\mathcal C\) and \(\|b_i\|\leq B\), then inside
		\(\mathcal C\rtimes_r\Gamma\) one has
		\[
		\left\|\frac1m\sum_{i=1}^m b_i u_{s_i t s_i^{-1}}\right\|
		\leq \frac{2B}{\sqrt m}.
		\]
	\end{lemma}
	
	\begin{proof}
		Use the faithful regular representation of the reduced crossed product on \(\ell^2(\Gamma,H)\) associated with a faithful representation of \(\mathcal C\) on \(H\), and keep the notation \(u_h\) for the represented unitaries. We use the convention that \(u_h\) sends \(\ell^2(Y,H)\) onto \(\ell^2(hY,H)\). The representation of \(\mathcal C\) is diagonal over the \(\Gamma\)-coordinate, hence it preserves the subspaces \(\ell^2(Y,H)\) for \(Y\subset\Gamma\) and commutes with the corresponding support projections.
		
		Put \(h_i=s_i t s_i^{-1}\), and let \(P_i\) be the projection onto \(\ell^2(s_iD,H)\) and \(Q_i=1-P_i\). Then the \(P_i\)'s are pairwise orthogonal. Since \(Q_i\) projects onto \(\ell^2(s_iC,H)\), the relation \(tC\cap C=\varnothing\) gives \(u_{h_i}\ell^2(s_iC,H)=\ell^2(s_itC,H)\perp\ell^2(s_iC,H)\). As \(b_i\) preserves supports in the \(\Gamma\)-coordinate, \(Q_i b_i u_{h_i}Q_i=0\). Hence, for \(T_i=b_i u_{h_i}\), one has \(T_i=(P_i+Q_i)T_i(P_i+Q_i)=P_iT_iP_i+P_iT_iQ_i+Q_iT_iP_i=P_iT_i+Q_iT_iP_i\). Thus, for unit vectors \(\zeta,\eta\in\ell^2(\Gamma,H)\),
		\[
		\left|\sum_i\langle b_i u_{h_i}\zeta,\eta\rangle\right|
		\leq
		B\sum_i\|P_i\eta\|+B\sum_i\|P_i\zeta\|
		\leq 2B\sqrt m.
		\]
		Dividing by \(m\) gives the estimate.
	\end{proof}
	
	\subsection{Coordinate slicing}
	
	For \(k\in N\) and \(\xi\in X_k\), write \(\sigma_{k,\xi}:\rC(X_N)\to\rC(X_{N\setminus\{k\}})\colon f\mapsto f_{k,\xi}\), where \(f_{k,\xi}((x_i)_{i\neq k})=f(x_1,\ldots,x_{k-1},\xi,x_{k+1},\ldots,x_n)\). Using the canonical inclusion \(\rC(X_{N\setminus\{k\}})\subset\rC(X_N)\), we regard every slice as a function on \(X_N\) that is independent of the \(k\)-th coordinate.
	
	The expectation need not preserve an intermediate algebra, so \(\rE(d)\) need not belong to \(\mathcal D\) even when \(d\in\mathcal D\). The next lemma nevertheless recovers a coordinate slice of \(\rE(d)\) as an actual diagonal element of \(\mathcal D\). The argument is reminiscent of \cite[Lemma~3.10]{Am21}, where Powers averaging removes the nonzero Fourier coefficients while the action on the diagonal has a prescribed limit. Here rank-one geometry makes this limit a coordinate slice rather than the full conditional expectation. The same conjugating elements are nearly trivial outside the \(k\)-th factor, which makes the diagonal average approach the slice at \(\xi\), while their disjoint contracting regions in the \(k\)-th factor give the Powers estimate for every nonzero Fourier coefficient.
	
	\begin{lemma}\label{lem:coordinate-slicing}
		Let \(k\in N\) and let \(\xi\in X_k\) have trivial stabilizer for the \(\Gamma\)-action on \(X_k\) through \(p_k\). If \(\mathcal D\) is an intermediate \(\rC^*\)-subalgebra of \(\mathcal A_N\) containing \(\rC^*_\lambda(\Gamma)\), then \(\sigma_{k,\xi}(\rE(\mathcal D))\subset\mathcal D\cap\rC(X_{N\setminus\{k\}})\).
	\end{lemma}
	
	\begin{proof}
		Fix \(d\in\mathcal D\) and \(\varepsilon>0\). Choose a finite Fourier polynomial \(p=f+\sum_{t\in F}f_tu_t\), where \(F\subset\Gamma\setminus\{e\}\), such that \(\|d-p\|<\varepsilon\). Then \(f=\rE(p)\) and \(\|\rE(d)-f\|<\varepsilon\).
		
		We first choose the geometric data for the average. The stabilizer hypothesis and the finiteness of \(F\) give a nonempty open neighborhood \(O\subset X_k\) of \(\xi\), with nonempty complement, such that
		\[
		p_k(t)\overline O\cap\overline O=\varnothing
		\qquad(t\in F).
		\]
		After shrinking \(O\), uniform continuity of \(f\) gives
		\[
		|f(y,z)-f(y,\xi)|<\varepsilon/2
		\]
		for every \(y\in X_{N\setminus\{k\}}\) and \(z\in O\). This preserves the separation condition and the nonemptiness of the complement. Continuity of the action on the complementary factors gives an identity neighborhood \(U\subset G_{N\setminus\{k\}}\) such that \(|f(g^{-1}y,z)-f(y,z)|<\varepsilon/2\) for every \(g\in U\) and \((y,z)\in X_N\). Hence
		\[
		|f(g^{-1}y,z)-f(y,\xi)|<\varepsilon
		\]
		for every \(g\in U\), \(y\in X_{N\setminus\{k\}}\), and \(z\in O\).
		
		Put \(M=\sum_{t\in F}\|f_t\|\). Choose \(m\) large enough that \(2M/\sqrt m<\varepsilon\) and \(2\|f\|/m<\varepsilon\), and choose pairwise disjoint nonempty open sets \(V_1,\ldots,V_m\subset X_k\), possible since \(X_k\) has no isolated points. Since \(X_k\setminus O\neq\varnothing\), the codimension-one assertion of Proposition~\ref{prop:codim-one-strip} gives \(s_1,\ldots,s_m\in\Gamma\) such that
		\[
		p_{N\setminus\{k\}}(s_i)\in U,
		\qquad
		p_k(s_i)(X_k\setminus O)\subset V_i
		\qquad(1\leq i\leq m).
		\]
		Thus the \(s_i\)'s move little outside the \(k\)-th factor and contract \(X_k\setminus O\) into disjoint regions. Define \(\Psi:\mathcal A_N\to\mathcal A_N\) by
		\[
		\Psi(a)=\frac1m\sum_{i=1}^m u_{s_i}au_{s_i}^*.
		\]
		Since \(\mathcal D\) contains \(\rC^*_\lambda(\Gamma)\), one has \(\Psi(\mathcal D)\subset\mathcal D\).
		
		We first use this geometry to identify the diagonal limit. Write \(x=(y,z)\in X_{N\setminus\{k\}}\times X_k\). At most one of the sets \(V_i\) contains \(z\). For every other \(i\), one has \(z\notin V_i\). The contrapositive of \(p_k(s_i)(X_k\setminus O)\subset V_i\) gives \(p_k(s_i)^{-1}z\in O\). Hence \((s_i\cdot f)(y,z)=f(p_{N\setminus\{k\}}(s_i)^{-1}y,p_k(s_i)^{-1}z)\) is within \(\varepsilon\) of \(\sigma_{k,\xi}(f)(y,z)=f(y,\xi)\). Therefore
		\(\left\|m^{-1}\sum_{i=1}^m s_i\cdot f-\sigma_{k,\xi}(f)\right\|
		\leq\varepsilon+2\|f\|/m<2\varepsilon\).
		
		We now use the same geometric choices to build a Powers partition. Set \(C=\{\gamma\in\Gamma\mid p_k(\gamma)\xi\in O\}\), and put \(D=\Gamma\setminus C\). If \(\gamma\in tC\cap C\), write \(\gamma=t\delta\) with \(\delta\in C\). Then \(p_k(\delta)\xi\in O\) and \(p_k(t)p_k(\delta)\xi=p_k(\gamma)\xi\in O\), contradicting \(p_k(t)O\cap O=\varnothing\). Thus \(tC\cap C=\varnothing\) for every \(t\in F\). If \(\gamma=s_i\delta\in s_iD\), then \(p_k(\delta)\xi\notin O\), so \(p_k(\gamma)\xi\in V_i\). Hence \(\gamma\in s_iD\cap s_jD\) implies \(p_k(\gamma)\xi\in V_i\cap V_j\), which is impossible when \(i\neq j\). Thus the sets \(s_iD\) are pairwise disjoint. Applying Lemma~\ref{lem:powers-criterion} to \(b_i=s_i\cdot f_t\) gives
		\(\left\|m^{-1}\sum_{i=1}^m(s_i\cdot f_t)u_{s_i t s_i^{-1}}\right\|
		\leq2\|f_t\|/\sqrt m\).
		Summing over \(t\in F\) gives \(\|\Psi(p)-\sigma_{k,\xi}(f)\|<3\varepsilon\). Finally,
		\[\|\Psi(d)-\sigma_{k,\xi}(\rE(d))\|
			\leq
			\|\Psi(d-p)\|+\|\Psi(p)-\sigma_{k,\xi}(f)\| +\|\sigma_{k,\xi}(f-\rE(d))\|
			<5\varepsilon.	\]
		Thus, for every \(\varepsilon>0\), the algebra \(\mathcal D\) contains an element within \(5\varepsilon\) of \(\sigma_{k,\xi}(\rE(d))\). Since \(\mathcal D\) is norm closed, the slice belongs to \(\mathcal D\). By construction it also belongs to \(\rC(X_{N\setminus\{k\}})\), which proves the assertion.
	\end{proof}
	
	We can now combine the Fourier upper bound with the commutative factor theorem. The former places an intermediate algebra below a coordinate crossed product. Coordinate slicing shows that every coordinate in this upper bound is already present in the diagonal of the intermediate algebra.
	
	\begin{proof}[Proof of Theorem~\ref{thm:classification}]
		Let \(\mathcal D\) be an intermediate \(\rC^*\)-subalgebra of \(\mathcal A_N\) containing \(\rC^*_\lambda(\Gamma)\). By Proposition~\ref{prop:dani-ap-upper}, there is a unique \(S\subset N\) such that \(\rC^*(\rE(\mathcal D))=\rC(X_S)\) and \(\mathcal D\subset\mathcal A_S\).
		
		The \(\rC^*\)-subalgebra \(\mathcal D\cap\rC(X_N)\) is unital and \(\Gamma\)-invariant, since \(\mathcal D\) contains the canonical unitaries. Hence Theorem~\ref{thm:dani-coordinate} gives a unique \(T\subset N\) such that \(\mathcal D\cap\rC(X_N)=\rC(X_T)\).
		The set \(S\) records the coordinates visible in the Fourier coefficients of \(\mathcal D\), while \(T\) records those already present in its diagonal part. We show that \(T=S\).
		
		Applying the canonical expectation shows that \(\mathcal A_S\cap\rC(X_N)=\rC(X_S)\). Since \(\mathcal D\subset\mathcal A_S\), it follows that \(\rC(X_T)\subset\rC(X_S)\), and hence \(T\subset S\).
		
		For the reverse inclusion, fix \(j\in S\). Since \(\rC^*(\rE(\mathcal D))=\rC(X_S)\), some \(f\in\rE(\mathcal D)\) depends on the \(j\)-th coordinate. Choose \(k\in N\setminus\{j\}\), which is possible since \(n\geq2\). The function \(f\) separates two points \(x^{(0)},x^{(1)}\in X_N\) that agree outside the \(j\)-th coordinate. Let \(\xi_0\) be their common \(k\)-th coordinate, and let \(x^{(\ell)}(\xi)\) be obtained from \(x^{(\ell)}\) by replacing that coordinate by \(\xi\). By continuity, there is an open neighborhood \(W\subset X_k\) of \(\xi_0\) such that
		\[
		f(x^{(0)}(\xi))\neq f(x^{(1)}(\xi))
		\qquad(\xi\in W).
		\]
		Choose \(\xi\in W\) with trivial \(\Gamma\)-stabilizer, using Lemma~\ref{lem:coordinate-action-topologically-free}. Lemma~\ref{lem:coordinate-slicing} gives \(\sigma_{k,\xi}(f)\in\mathcal D\cap\rC(X_N)=\rC(X_T)\). This slice still depends on the \(j\)-th coordinate, so \(j\in T\). Thus \(S\subset T\), and consequently \(S=T\).
		
		It follows that \(\rC(X_S)=\mathcal D\cap\rC(X_N)\subset\mathcal D\). Together with the canonical unitaries, which already belong to \(\mathcal D\), this generates \(\mathcal A_S\). Thus \(\mathcal A_S\subset\mathcal D\). The reverse inclusion was supplied by Proposition~\ref{prop:dani-ap-upper}, so \(\mathcal D=\mathcal A_S\). Along with the defining equality
		\(\rC^*(\rE(\mathcal D))=\rC(X_S)\), this also proves the more precise assertion.
		
		Each \(\mathcal A_S\) is plainly an intermediate \(\rC^*\)-subalgebra, and the preceding argument proves surjectivity. If \(R\subset S\), then \(\mathcal A_R\subset\mathcal A_S\). Conversely, if \(\mathcal A_R\subset\mathcal A_S\), applying \(\rE\) gives \(\rC(X_R)\subset\rC(X_S)\), which forces \(R\subset S\). Thus the map and its inverse preserve inclusion.
	\end{proof}
	
	\subsection{The real rank-one Lie-group case}
	
	\begin{proof}[Proof of Corollary~\ref{cor:dani-real-rank-one-direct}]
		For every $i$, the space $X_i=G_i/P_i$ is infinite, compact, and metrizable.
		The rank-one Bruhat decomposition implies that $G_i$ acts transitively on
		$X_i^{(2)}$ \cite[Theorem~4.9]{Kn97}.  A regular element in a maximal
		$\R$-split torus has north--south dynamics on $X_i$
		\cite[Sections~3.3 and~3.5]{Be97}. Thus $G_i\curvearrowright X_i$ is a
		rank-one boundary action. The action is faithful because $G_i$ is
		center-free, and irreducibility gives density of $p_i(\Gamma)$ in $G_i$.
		
		We next verify topological freeness. Fix \(i\in N\), put
		\(H=G_{N\setminus\{i\}}\), and let \(K=\ker(p_i|_\Gamma)=\Gamma\cap H\).
		The discrete subgroup \(K<H\) is normalized by the dense subgroup
		\(p_{N\setminus\{i\}}(\Gamma)\). Since its normalizer is closed, \(K\lhd H\).
		For \(k\in K\), the continuous map \(H\to K:h\mapsto hkh^{-1}\) is constant
		because \(H\) is connected and \(K\) is discrete. Hence \(K\subset\mathcal Z(H)\),
		which is trivial, and \(p_i|_\Gamma\) is injective. If a nontrivial element of
		\(G_i\) fixed a nonempty open subset of \(X_i\), the real-analytic identity
		principle would force it to act trivially on all of \(X_i\), contradicting
		faithfulness. Thus every coordinate action \(\Gamma\curvearrowright X_i\) is
		topologically free.
		
		Finally, every \(G_i\) is weakly amenable by \cite[Main theorem]{CH89}, and
		hence has {\rm(AP)} by \cite[Theorem~1.12]{HK94}. The approximation property is
		stable under finite direct products by \cite[Theorem~1.15]{HK94}, so \(G\) has
		{\rm(AP)}. Its lattice \(\Gamma\) has
		{\rm(AP)} by \cite[Theorem~2.4]{HK94}. All the hypotheses of
		Theorem~\ref{thm:classification} are therefore satisfied, and
		Corollary~\ref{cor:dani-real-rank-one-direct} follows.
	\end{proof}
	
	The commutative part of Corollary~\ref{cor:dani-real-rank-one-direct} is also a direct consequence of Dani's
	theorem \cite[Theorem~A]{Da84}, which applies more generally to irreducible
	lattices in semisimple real Lie groups of total real rank at least two.
	
	\section{Further applications to rank-one product groups}
	\label{sec:applications}
	
	We now record two further applications of Theorem~\ref{thm:classification}, to products of trees
	and to mixed products over local fields.
	
	\subsection{Products of trees}
	
	Let $T$ be a locally finite tree with infinitely many ends.  We endow
	$\Aut(T)$ with the compact-open topology and $\partial T$ with the usual end
	topology.  A hyperbolic automorphism $a\in\Aut(T)$ has an axis and two fixed
	ends $a^-,a^+\in\partial T$, and it acts on $\partial T$ with north--south
	dynamics.
	
	\begin{theorem}
		\label{thm:dani-products-trees}
		Let $n\geq2$.  For every $i\in N$, let $T_i$ be a locally finite tree with
		infinitely many ends and let $G_i<\Aut(T_i)$ be a second countable closed
		subgroup. Assume the following conditions.
		\begin{enumerate}[\rm (i)]
			\item $G_i$ acts transitively on ordered pairs of distinct ends of $T_i$.
			\item $G_i$ contains a hyperbolic element.
		\end{enumerate}
		Let $\Gamma<G_1\times\cdots\times G_n$ be a lattice such that, for each $i$,
		the image of $\Gamma$ in $G_i/K_i$ is dense, where
		$K_i=\ker(G_i\curvearrowright\partial T_i)$. Assume,
		in addition, that every coordinate action
		$\Gamma\curvearrowright\partial T_i$ is topologically free. Put
		$X_i=\partial T_i$. Then
		\[
		2^N\longrightarrow
		\left\{\mathcal D\mid
		\rC^*_{\lambda}(\Gamma)\subset\mathcal D\subset
		\rC(X_N)\rtimes_r\Gamma\right\},
		\qquad
		S\longmapsto\rC(X_S)\rtimes_r\Gamma,
		\]
		is an order isomorphism.
	\end{theorem}
	
	\begin{proof}
		Each $\partial T_i$ is an infinite compact metrizable space.  A hyperbolic
		element of $G_i$ has north--south dynamics on $\partial T_i$, and transitivity
		on ordered pairs of distinct ends allows its repelling and attracting ends to
		be moved to any prescribed ordered pair. Hence every action
		$G_i\curvearrowright\partial T_i$ is a rank-one boundary action, and the
		density hypothesis is the effective-density assumption of Theorem~\ref{thm:classification}.
		
		For a closed subgroup of the automorphism group of a locally finite tree,
		vertex stabilizers are compact. Szwarc's theorem therefore shows that every
		$G_i$ is weakly amenable \cite[Theorem~6]{Sz91}, hence has {\rm(AP)} by
		\cite[Theorem~1.12]{HK94}. The approximation property is stable under finite
		direct products by \cite[Theorem~1.15]{HK94}, and the lattice $\Gamma$ inherits
		{\rm(AP)}
		\cite[Theorem~2.4]{HK94}. The assumed topological freeness supplies the last
		hypothesis of Theorem~\ref{thm:classification}.
	\end{proof}
	
	\begin{remark}
		\label{rem:tree-topological-freeness}
		Without topological freeness, Theorem~\ref{thm:abstract-topological-factor}
		still gives the commutative classification, but the crossed-product conclusion
		need not follow. Indeed, general tree automorphism groups may contain nontrivial
		elements that fix pointwise the set of ends of a half-tree.
	\end{remark}

	A useful criterion is available for uniform lattices in products of biregular
	trees. Let \(T_1,\ldots,T_n\) be
	thick locally finite biregular trees and let
	\(\Gamma<\Aut^+(T_1)\times\cdots\times\Aut^+(T_n)\) be a cocompact lattice,
	where \(\Aut^+(T_i)\) is the subgroup preserving the natural bicoloring. Then
	the action \(\Gamma\curvearrowright\partial T_i\) is topologically free if and
	only if the coordinate projection \(p_i|_\Gamma\) is injective.

	Necessity is immediate. Conversely, \cite[Proposition~6.4]{BBHP22} shows that
	injectivity makes the fixed-point set of every nonidentity element null for the
	unique $\Aut^+(T_i)$-invariant boundary measure class. A standard ray measure
	based at a vertex assigns positive mass to every nonempty boundary shadow, so
	this measure class has full support. Every such fixed-point set therefore has
	empty interior.

	\begin{corollary}
		Let \(\Gamma<\Aut^+(T_1)\times\Aut^+(T_2)\) be one of the finitely
		presented torsion-free simple cocompact Burger--Mozes lattices
		\cite{BM00b}, and put
		\(X=\partial T_1\times\partial T_2\). Then the \(\rC^*\)-subalgebras
		\(\mathcal D\) satisfying
		\(\rC^*_{\lambda}(\Gamma)\subset\mathcal D\subset\rC(X)\rtimes_r\Gamma\)
		are precisely
		\[
		\rC^*_{\lambda}(\Gamma),\qquad
		\rC(\partial T_1)\rtimes_r\Gamma,\qquad
		\rC(\partial T_2)\rtimes_r\Gamma,\qquad
		\rC(X)\rtimes_r\Gamma.
		\]
	\end{corollary}

	\begin{proof}
		Let \(d_i\geq6\) be the degree of \(T_i\), and put
		\(G_i=U(\operatorname{Alt}(d_i))^+<\Aut^+(T_i)\). By construction,
		\(\Gamma<G_1\times G_2\).
		The Burger--Mozes results show that \(p_i(\Gamma)\) is dense in \(G_i\) for
		\(i=1,2\) \cite[Theorem~5.5]{BM00b}. The universal groups and their index-two
		subgroups are described in
		\cite[Section~3.2 and Proposition~3.2.1]{BM00a}. Since
		\(\operatorname{Alt}(d_i)\) is \(2\)-transitive,
		\(U(\operatorname{Alt}(d_i))\) is locally \(\infty\)-transitive by
		\cite[Section~3.2]{BM00a}. Its plus subgroup \(G_i\) retains this property by
		\cite[Proposition~3.1.2(2)]{BM00a}, and hence acts transitively on ordered pairs
		of distinct ends by \cite[Lemma~3.1.1]{BM00a}. It also contains hyperbolic
		automorphisms.

		Each coordinate projection \(p_i|_\Gamma\) is nontrivial and hence injective,
		because its kernel is normal and \(\Gamma\) is simple. By construction,
		\(\Gamma\) acts freely and cocompactly on \(T_1\times T_2\), and hence is a
		cocompact lattice in \(\Aut^+(T_1)\times\Aut^+(T_2)\). Therefore,
		the topological-freeness criterion above shows that both
		coordinate boundary actions are topologically free. All the hypotheses of
		Theorem~\ref{thm:dani-products-trees} are therefore satisfied. The corresponding
		coordinate \(\rC^*\)-subalgebras are exactly the four displayed above.
	\end{proof}
	
	\subsection{Products of rank-one algebraic groups over local fields}
	
	For each $i\in N$, let $k_i$ be a local field and let $\mathbf G_i$ be a
	connected adjoint $k_i$-simple algebraic $k_i$-group of $k_i$-rank one. Put
	$G_i=\mathbf G_i(k_i)^+$, where $\mathbf G_i(k_i)^+$ is the subgroup generated
	by $\mathbf U(k_i)$ as $\mathbf U$ ranges over the unipotent $k_i$-split
	subgroups of $\mathbf G_i$ \cite[Proposition~I.5.4(i)]{Ma91}.
	Choose a minimal $k_i$-parabolic subgroup $\mathbf P_i<\mathbf G_i$, put
	$P_i=G_i\cap\mathbf P_i(k_i)$ and set $X_i=G_i/P_i$. The equality
	$\mathbf G_i(k_i)=G_i\mathbf P_i(k_i)$ identifies $X_i$ with the compact
	rank-one spherical building $\mathbf G_i(k_i)/\mathbf P_i(k_i)$
	\cite[Proposition~I.5.4(vi)]{Ma91}.
	When $k_i$ is non-archimedean, the reduced Bruhat--Tits building of $G_i$ is a
	locally finite tree and $X_i$ identifies $G_i$-equivariantly with its space of
	ends \cite{BT72}.
	
	\begin{theorem}
		\label{thm:dani-mixed-local-fields}
		Let $n\geq2$, put $G=\prod_{i\in N}G_i$ and let $\Gamma<G$ be an
		irreducible lattice, meaning that its projection to every nonempty proper
		subproduct is dense. Then
		\[
		2^N\longrightarrow
		\left\{\mathcal D\mid
		\rC^*_{\lambda}(\Gamma)\subset\mathcal D\subset
		\rC(X_N)\rtimes_r\Gamma\right\},
		\qquad
		S\longmapsto\rC(X_S)\rtimes_r\Gamma,
		\]
		is an order isomorphism.
	\end{theorem}
	
	\begin{proof}
		Each $X_i$ is infinite, compact, and metrizable. Fix $i\in N$, and let
		$\mathbf U_i$ be the unipotent radical of $\mathbf P_i$. The rank-one Bruhat
		decomposition identifies the big cell $X_i\setminus\{P_i\}$ with
		$\mathbf U_i(k_i)$, on which $\mathbf U_i(k_i)$ acts simply transitively
		\cite[Theorem~5.15]{BT65}. Hence $G_i$ is transitive on ordered pairs of
		distinct points. A maximal $k_i$-split torus contains an element $a_i$ with
		$|\alpha_i(a_i)|\neq1$ \cite[Theorem~7.2]{BT65}. It has north--south dynamics
		on $X_i$, by \cite[Sections~3.3 and~3.5]{Be97} in the archimedean case and by
		its hyperbolic action on the Bruhat--Tits tree in the non-archimedean case
		\cite{BT72}. Thus $G_i\curvearrowright X_i$ is a rank-one boundary action.
		Adjointness makes this action faithful, and irreducibility makes
		$p_i(\Gamma)$ dense in $G_i$.

		We verify topological freeness. Put $H=G_{N\setminus\{i\}}$ and
		$K=\ker(p_i|_\Gamma)=\Gamma\cap H$. The discrete subgroup $K<H$ is
		normalized by the dense subgroup $p_{N\setminus\{i\}}(\Gamma)$. Its
		normalizer is closed, so $K\lhd H$. For every $j\neq i$, the group
		$K\cap G_j$ is a discrete normal subgroup of the nondiscrete topologically
		simple group $G_j$, and hence is trivial
		\cite[Theorem~I.1.5.6(i)]{Ma91}. Since $[K,G_j]\subset K\cap G_j$, the group
		$K$ centralizes every factor of $H$. These factors are center-free, so
		$K=\{e\}$ and $p_i|_\Gamma$ is injective.

		For $g\in G_i\setminus\{e\}$, the fixed locus of $g$ in
		$\mathbf G_i/\mathbf P_i$ is proper and Zariski closed. It has empty interior,
		since every nonempty open subset of $X_i$ meets the big Bruhat cell in a
		nonempty $k_i$-open subset of $\mathbf U_i(k_i)$, which is Zariski dense in
		$\mathbf U_i$. Together with the injectivity of $p_i|_\Gamma$, this proves that
		$\Gamma\curvearrowright X_i$ is topologically free.
		
		The archimedean factors are weakly amenable by \cite[Main theorem]{CH89}, and
		the non-archimedean factors are weakly amenable by \cite[Theorem~6]{Sz91},
		applied to their Bruhat--Tits trees. Thus every $G_i$ has {\rm(AP)}
		\cite[Theorem~1.12]{HK94}. Finite products preserve {\rm(AP)}, and lattices
		inherit it \cite[Theorems~1.15 and~2.4]{HK94}. Theorem~\ref{thm:classification}
		now applies.
	\end{proof}
	
	\begin{example}
		Let $p$ be a prime and diagonally embed $\Gamma=\PSL_2(\Z[1/p])$ in
		$\PSL_2(\R)\times\PSL_2(\Q_p)$.
		These two ambient groups are the plus groups associated with the adjoint
		algebraic group $\PGL_2$.
		The $S$-arithmetic Borel--Harish-Chandra theorem makes $\Gamma$ a lattice
		\cite[Theorem~I.3.2.5]{Ma91}, and strong approximation makes its two projections
		dense \cite[\S II.6.8]{Ma91}. In particular, the lattice is
		irreducible, so Theorem~\ref{thm:dani-mixed-local-fields} applies. For the
		boundary action on
		$X=\mathbb P^1(\R)\times\mathbb P^1(\Q_p)$, the intermediate
		$\rC^*$-subalgebras are exactly
		\[
		\rC^*_\lambda(\Gamma),\qquad
		\rC(\mathbb P^1(\R))\rtimes_r\Gamma,\qquad
		\rC(\mathbb P^1(\Q_p))\rtimes_r\Gamma,\qquad
		\rC(X)\rtimes_r\Gamma.
		\]
		The same proof applies when
		$\PSL_2\left(\Z\left[1/(p_1\cdots p_r)\right]\right)$ is diagonally embedded in
		$\PSL_2(\R)\times\prod_{j=1}^r\PSL_2(\Q_{p_j})$, where
		$p_1,\ldots,p_r$ are distinct primes. Its boundary has $r+1$
		coordinates, and there are exactly $2^{r+1}$ intermediate $\rC^*$-subalgebras.
	\end{example}
	
	\begin{remark}
		The projective quotient in the preceding example is essential. For
		$\SL_2(\Z[1/p])$, the central element $-I$ acts trivially on both boundary
		coordinates. Write
		\(\widetilde{\mathcal A}_S=\rC(X_S)\rtimes_r\SL_2(\Z[1/p])\) and put
		\(e_\pm=(1\pm u_{-I})/2\). For \(S\neq T\),
		\(e_+\widetilde{\mathcal A}_S\oplus
		e_-\widetilde{\mathcal A}_T\) is an intermediate \(\rC^*\)-subalgebra but is not listed by a single
		subset of \(N\). This is the reduced crossed-product analogue of the von Neumann
		algebra obstruction noted in \cite[Remark following the proof of Theorem~A]{BH23}.
	\end{remark}
	
	\section{ITAP as a necessary condition}\label{sec:scalar-itap}
	
	We conclude with an obstruction to extending Theorem~\ref{thm:classification} beyond products of rank-one groups. Although the ambient groups in Corollary~\ref{cor:dani-real-rank-one-direct} can have higher total real rank, their rank-one factors ensure that the lattices considered there have {\rm(AP)}, and hence {\rm ITAP}. For lattices in connected semisimple real Lie groups, the case of scalar expectation in the noncommutative Dani problem is governed precisely by the invariant translation approximation property.
	
	\subsection{A general observation}
	
	Let \(\Lambda\) be a countable discrete group acting on a compact space \(X\), put \(\mathcal A=\rC(X)\rtimes_r\Lambda\), and denote by \(\rE:\mathcal A\to\rC(X)\) the canonical expectation. Suppose that \(x_0\in X\) has dense \(\Lambda\)-orbit. Define a covariant representation on \(\ell^2(\Lambda)\) by \(\pi_{x_0}(f)\delta_t=f(tx_0)\delta_t\) and \(\pi_{x_0}(u_s)=\lambda_s\).
	Here our action convention is \((s\cdot f)(x)=f(s^{-1}x)\), and covariance follows from
	\[
	\lambda_s\pi_{x_0}(f)\lambda_s^*\delta_t
	=f(s^{-1}tx_0)\delta_t
	=\pi_{x_0}(s\cdot f)\delta_t.
	\]
	This is the regular covariant representation induced by evaluation at \(x_0\). Indeed, let \(\sigma\) be a faithful representation of \(\rC(X)\). The regular representation induced by \(\sigma\oplus\operatorname{ev}_{x_0}\) is the direct sum of the one induced by \(\sigma\) and \(\pi_{x_0}\). Since \(\sigma\oplus\operatorname{ev}_{x_0}\) is faithful, \cite[Proposition~4.1.5]{BO08} gives
	\[
		\|\pi_{x_0}(p)\|\leq\|p\|_r
	\]
	for every finite Fourier polynomial \(p\). Thus \(\pi_{x_0}\) integrates to a representation of \(\mathcal A\). It acts on \(\ell^2(\Lambda)\), independently of the stabilizer of \(x_0 \in X\).
	
	\begin{lemma}\label{lem:itap-orbit-representation}
		The integrated representation \(\pi_{x_0}:\mathcal A\to\rB(\ell^2(\Lambda))\) is faithful. For every \(a\in\mathcal A\) and \(r,t\in\Lambda\),
		\begin{equation}\label{eq:itap-matrix-coefficient}
			\langle\pi_{x_0}(a)\delta_t,\delta_r\rangle
			=
			\rE(au_{rt^{-1}}^*)(rx_0).
		\end{equation}
		Consequently,
		\begin{equation}\label{eq:itap-boundary-intersection}
			\pi_{x_0}\bigl(\{a\in\mathcal A\mid
			\rE(au_s^*)\in\C1\text{ for every }s\in\Lambda\}\bigr)
			=
			\pi_{x_0}(\mathcal A)\cap\rL(\Lambda).
		\end{equation}
	\end{lemma}
	
	\begin{proof}
		For a finite Fourier polynomial \(p=\sum_g f_gu_g\), one has
		\[
		\langle\pi_{x_0}(p)\delta_t,\delta_r\rangle
		=f_{rt^{-1}}(rx_0)
		=\rE(pu_{rt^{-1}}^*)(rx_0).
		\]
		Both sides depend continuously on \(p\), which proves
		\eqref{eq:itap-matrix-coefficient}. In particular, for \(a\in\mathcal A_+\)
		and \(t\in\Lambda\), the case \(r=t\) gives
		\[
		\langle\pi_{x_0}(a)\delta_t,\delta_t\rangle=\rE(a)(tx_0).
		\]
		If \(\pi_{x_0}(a)=0\), density of \(\Lambda x_0\) gives \(\rE(a)=0\).
		Faithfulness of the canonical expectation then gives \(a=0\). If
		\(\pi_{x_0}(a)=0\) for arbitrary \(a\in\mathcal A\), then
		\(\pi_{x_0}(a^*a)=0\), so the positive case gives \(a^*a=0\) and hence
		\(a=0\). Thus \(\pi_{x_0}\) is faithful.

		If \(\rho_q\delta_t=\delta_{tq^{-1}}\), then
		\[
		T\in\rL(\Lambda)=\rho(\Lambda)'
		\quad\Longleftrightarrow\quad
		\langle T\delta_t,\delta_r\rangle=c_{rt^{-1}}
		\quad(r,t\in\Lambda)
		\]
		for a scalar family \((c_s)_{s\in\Lambda}\). Thus scalar Fourier coefficients imply that \(\pi_{x_0}(a)\in\rL(\Lambda)\). Conversely, if \(\pi_{x_0}(a)\in\rL(\Lambda)\), fix \(s\in\Lambda\) and set \(t=s^{-1}r\) in \eqref{eq:itap-matrix-coefficient}. Then \(\rE(au_s^*)(rx_0)\) is independent of \(r\), so \(\rE(au_s^*)\) is constant on the dense orbit \(\Lambda x_0\), and hence on \(X\). This proves \eqref{eq:itap-boundary-intersection}.
	\end{proof}
	
	We use the convention
	\[
	\rC_u^*(\Lambda)
	=
	\overline{\operatorname{span}}
	\{M_b\lambda_s\mid b\in\ell^\infty(\Lambda),\ s\in\Lambda\}
	\subset\rB(\ell^2(\Lambda)),
	\]
	where \(M_b\) denotes multiplication by \(b\). For the left-translation action of \(\Lambda\) on \(\ell^\infty(\Lambda)\), this is precisely the canonical reduced-crossed-product representation, and hence
	\(\rC_u^*(\Lambda)=\ell^\infty(\Lambda)\rtimes_r\Lambda\).
	Equivalently, \(\rC_u^*(\Lambda)\) is the norm closure of the operators supported
	on finitely many sets
	\[
	\Delta_s=\{(r,t)\in\Lambda\times\Lambda\mid r=st\}
	=\{(r,t)\in\Lambda\times\Lambda\mid rt^{-1}=s\},
	\qquad s\in\Lambda.
	\]
	Each \(\Delta_s\) is the graph of left translation by \(s\) and is invariant
	under simultaneous right translation of both coordinates. We call these sets the
	right-invariant diagonals. The group \(\Lambda\) has the invariant translation
	approximation property, abbreviated ITAP \cite{Za06}, when
	\begin{equation}
		\rC_u^*(\Lambda)\cap\rL(\Lambda)
		=
		\rC^*_{\lambda}(\Lambda).
	\end{equation}
	With the right regular representation \(\rho\) as above, one has \(\rC_u^*(\Lambda)\cap\rL(\Lambda)=\rC_u^*(\Lambda)^{\Ad(\rho(\Lambda))}\).
	
	By definition, \(\pi_{x_0}(f)\) is the multiplication operator \(M_{b_f}\), where \(b_f(t)=f(tx_0)\), while \(\pi_{x_0}(u_s)=\lambda_s\). Thus \(\pi_{x_0}(fu_s)=M_{b_f}\lambda_s\in\rC_u^*(\Lambda)\) for every \(f\in\rC(X)\) and \(s\in\Lambda\). Since finite Fourier sums are dense in \(\mathcal A\), it follows that \(\pi_{x_0}(\mathcal A)\subset\rC_u^*(\Lambda)\). The faithfulness established in Lemma~\ref{lem:itap-orbit-representation} makes this an embedding.

	For comparison, Suzuki shows that for every exact discrete group \(\Lambda\), an
	intermediate \(\rC^*\)-subalgebra between \(\rC^*_{\lambda}(\Lambda)\) and
	\(\rC_u^*(\Lambda)\cap\rL(\Lambda)\) can be realized as a decreasing intersection
	of \(\rC^*\)-algebras isomorphic to \(\mathcal O_2\)
	\cite[Theorem~A]{Su17}. The point of the next proposition is that topological
	amenability together with a dense orbit forces saturation of this upper bound.
	
	\begin{proposition}\label{prop:amenable-orbit-saturation}
		Suppose that the action \(\Lambda\curvearrowright X\) is topologically amenable and that \(\Lambda x_0\) is dense in \(X\). Then, inside \(\rB(\ell^2(\Lambda))\),
		\begin{equation}
			\pi_{x_0}(\mathcal A)\cap\rL(\Lambda)
			=
			\rC_u^*(\Lambda)\cap\rL(\Lambda).
		\end{equation}
	\end{proposition}
	
	\begin{proof}
		By topological amenability and \cite[Lemma~4.3.8]{BO08}, there are continuous maps \(m_i:X\longrightarrow\Prob(\Lambda)\), where \(\Prob(\Lambda)\) carries the weak-$\ast$ topology, and finite sets \(F_i\subset\Lambda\) such that \(\supp m_i(x)\subset F_i\) for every \(x\in X\) and, for every \(s\in\Lambda\),
		\[
		\sup_{x\in X}
		\|m_i(x)-s\mathbin{\cdot}m_i(s^{-1}x)\|_1
		\longrightarrow0,
		\qquad
		(s\mathbin{\cdot}p)(t)=p(s^{-1}t).
		\]
		Since \(m_i(X)\subset\Prob(F_i)\), weak-$\ast$ continuity of \(m_i\) is equivalent to continuity for the \(\ell^1\)-norm. Put \(\xi_i(x)=m_i(x)^{1/2}\). These are continuous unit-vector fields from \(X\) to \(\ell^2(\Lambda)\), with \(\supp\xi_i(x)\subset F_i\), and
		\[
		\sup_{x\in X}
		\|\xi_i(x)-\lambda_s\xi_i(s^{-1}x)\|_2^2
		\leq
		\sup_{x\in X}
		\|m_i(x)-s\mathbin{\cdot}m_i(s^{-1}x)\|_1
		\longrightarrow0
		\]
		for every \(s\in\Lambda\).
		
		For \(r\in\Lambda\), put \(v_i(r)=\lambda_{r^{-1}}\xi_i(rx_0)\) and \(k_i(r,t)=\langle v_i(r),v_i(t)\rangle\). The kernel \(k_i\) is positive definite, satisfies \(k_i(r,r)=1\), and is supported on \(\{(r,t)\in\Lambda\times\Lambda\mid rt^{-1}\in F_iF_i^{-1}\}\).
		Indeed, \(\operatorname{supp}v_i(r)\subset r^{-1}F_i\), so \(k_i(r,t)\neq0\) implies that \(r^{-1}a=t^{-1}b\) for some \(a,b\in F_i\), and hence \(rt^{-1}=ab^{-1}\in F_iF_i^{-1}\).
		By \cite[Theorem~D.3]{BO08}, Schur multiplication by \(k_i\) defines a completely positive contraction \(\Theta_i\) on \(\rB(\ell^2(\Lambda))\). Since \(k_i(r,r)=1\) for every \(r\in\Lambda\), the map \(\Theta_i\) is unital.
		
		For \(h_i(s,x)=\langle\xi_i(x),\lambda_s\xi_i(s^{-1}x)\rangle\), one has
		\begin{equation}\label{eq:itap-schur-coefficient}
			k_i(r,t)=h_i(rt^{-1},rx_0),
			\qquad
			\sup_{x\in X}|h_i(s,x)-1|\longrightarrow0
		\end{equation}
		for every fixed \(s\in\Lambda\), and \(h_i(s,\mathord\cdot)=0\) unless \(s\in F_iF_i^{-1}\).
		
		We first claim that \(\Theta_i(T)\to T\) in norm for every \(T\in\rC_u^*(\Lambda)\). Fix \(s\in\Lambda\) and \(b\in\ell^\infty(\Lambda)\). The operator \(M_b\lambda_s\) is supported on the \(s\)-diagonal, with \((r,t)\)-entry \(b(r)\) when \(r=st\), and zero otherwise.
		Schur multiplication and \eqref{eq:itap-schur-coefficient} therefore give
		\begin{align*}
			\langle\Theta_i(M_b\lambda_s)\delta_t,\delta_r\rangle
			&=k_i(r,t)\langle M_b\lambda_s\delta_t,\delta_r\rangle\\
			&=k_i(r,t)b(r)\mathbf 1_{\{r=st\}}\\
			&=b(r)h_i(s,rx_0)\mathbf 1_{\{r=st\}},
		\end{align*}
		because \(r=st\) implies \(rt^{-1}=s\). These are precisely the matrix coefficients of \(M_{b(\mathord\cdot)h_i(s,(\mathord\cdot)x_0)}\lambda_s\),
		so
		\[
			\Theta_i(M_b\lambda_s)
			=M_{b(\mathord\cdot)h_i(s,(\mathord\cdot)x_0)}\lambda_s.
		\]
		Consequently, if \(B=\sum_{s\in S}M_{b_s}\lambda_s\), where \(S\subset\Lambda\) is finite and \(b_s\in\ell^\infty(\Lambda)\) for every \(s\in S\), then
		\[
		\|\Theta_i(B)-B\|
		\leq
		\sum_{s\in S}\|b_s\|_\infty
		\sup_{x\in X}|h_i(s,x)-1|
		\longrightarrow0.
		\]
		The claim follows by density and contractivity.
		
		Now take \(T\in\rC_u^*(\Lambda)\cap\rL(\Lambda)\) and put \(c_s=\langle T\delta_e,\delta_s\rangle\). Since the matrix entries of \(T\) are constant on the right-invariant diagonals \(\Delta_s\), \eqref{eq:itap-schur-coefficient} gives
		\begin{align*}
			\langle\Theta_i(T)\delta_t,\delta_r\rangle
			&=k_i(r,t)\langle T\delta_t,\delta_r\rangle\\
			&=c_{rt^{-1}}h_i(rt^{-1},rx_0).
		\end{align*}
		The right-hand side has finite diagonal support. Since each \(h_i(s,\mathord\cdot)\) belongs to \(\rC(X)\), the finite Fourier polynomial \(a_i=\sum_{s\in F_iF_i^{-1}}c_sh_i(s,\mathord\cdot)u_s\) belongs to \(\mathcal A\). By Lemma~\ref{lem:itap-orbit-representation}, \(\pi_{x_0}(a_i)\) has the same matrix coefficients as \(\Theta_i(T)\). Indeed, only the coefficient indexed by \(s=rt^{-1}\) contributes, and both sides vanish when \(rt^{-1}\notin F_iF_i^{-1}\) because \(h_i(rt^{-1},\mathord\cdot)=0\). Hence \(\Theta_i(T)=\pi_{x_0}(a_i)\in\pi_{x_0}(\mathcal A)\).
		Together with the preceding norm convergence and closedness of \(\pi_{x_0}(\mathcal A)\), this proves \(\rC_u^*(\Lambda)\cap\rL(\Lambda)\subset\pi_{x_0}(\mathcal A)\cap\rL(\Lambda)\).
		The reverse inclusion follows from \(\pi_{x_0}(\mathcal A)\subset\rC_u^*(\Lambda)\).
	\end{proof}
	
	We can now formulate the exact scalar obstruction.
	
	\begin{proposition}\label{prop:scalar-itap-equivalence}
		Suppose that \(\Lambda\curvearrowright X\) is topologically amenable and has a dense orbit. The following are equivalent.
		\begin{enumerate}[\rm (i)]
			\item Every intermediate \(\rC^*\)-subalgebra \(\rC^*_{\lambda}(\Lambda)\subset\mathcal D\subset\rC(X)\rtimes_r\Lambda\) satisfying \(\rC^*(\rE(\mathcal D))=\C1\) is equal to \(\rC^*_{\lambda}(\Lambda)\).
			\item The group \(\Lambda\) has {\rm ITAP}.
		\end{enumerate}
	\end{proposition}
	
	\begin{proof}
		Fix a point \(x_0\in X\) with dense orbit and identify \(\mathcal A\) with \(\pi_{x_0}(\mathcal A)\).
		Put \(\mathcal D_0=\mathcal A\cap\rL(\Lambda)\), which is an intermediate \(\rC^*\)-subalgebra. For any intermediate \(\rC^*\)-subalgebra \(\mathcal D\), Lemma~\ref{lem:itap-orbit-representation} gives \(\rC^*(\rE(\mathcal D))=\C1\Longleftrightarrow\mathcal D\subset\mathcal D_0\).
		Indeed, if the left-hand side holds, then \(du_s^*\in\mathcal D\) for \(d\in\mathcal D\) and \(s\in\Lambda\), so every Fourier coefficient \(\rE(du_s^*)\) is scalar and \(d\in\mathcal D_0\). Conversely, if \(\mathcal D\subset\mathcal D_0\), then \(\rE(\mathcal D)\subset\C1\), and unitality gives the left-hand side.
		
		Thus (i) is equivalent to \(\mathcal D_0=\rC^*_{\lambda}(\Lambda)\), since \(\mathcal D_0\) is itself an intermediate \(\rC^*\)-subalgebra. Proposition~\ref{prop:amenable-orbit-saturation} identifies
		\[
		\mathcal D_0
		=
		\rC_u^*(\Lambda)\cap\rL(\Lambda),
		\]
		so this equality is precisely {\rm ITAP}.
	\end{proof}
	
	\subsection{Consequence for the noncommutative Dani problem}
	
	Let \(G\) be a connected semisimple real Lie group with finite center and
	without compact factors, let \(P<G\) be a minimal parabolic subgroup, and let
	\(\Gamma<G\) be a lattice. The action \(G\curvearrowright G/P\) is a boundary
	action, and its restriction to \(\Gamma\) remains a boundary action
	\cite[p.~175, items~(5)--(6)]{Fu03}. In particular, every \(\Gamma\)-orbit is
	dense. See also \cite[Lemma~8.5]{Mo73}. The action
	\(\Gamma\curvearrowright G/P\) is topologically amenable by
	\cite[Theorem~6.6]{Ka05} and \cite[Theorem~5.4.1]{BO08}.
	Proposition~\ref{prop:scalar-itap-equivalence} therefore applies.
	
	\begin{proof}[Proof of Theorem~\ref{thm:semisimple-itap-obstruction}]
		The preceding paragraph verifies the hypotheses of
		Proposition~\ref{prop:scalar-itap-equivalence} for \(X=G/P\). Thus
		conditions~{\rm (i)} and~{\rm (ii)} are equivalent.
		
		Now assume the full noncommutative Dani classification and let \(\mathcal D\)
		be an intermediate \(\rC^*\)-subalgebra with \(\rC^*(\rE(\mathcal D))=\C1\). There is a
		parabolic subgroup \(P\subset Q\subset G\) such that
		\(\mathcal D=\rC(G/Q)\rtimes_r\Gamma\). Since
		\(\rC(G/Q)\subset\mathcal D\), the scalar-expectation condition forces \(G/Q\)
		to be a point. Hence \(Q=G\) and
		\(\mathcal D=\rC^*_{\lambda}(\Gamma)\). The already proved equivalence gives
		{\rm ITAP}.
	\end{proof}
	
	The theorem applies to reducible as well as irreducible lattices. By the
	equivalence above, {\rm ITAP} gives exactly the upper bound for intermediate
	\(\rC^*\)-subalgebras with scalar expectation. It does not by itself settle the
	other parabolic cases.
	
	For \(\SL_3(\Z)\), the failure of {\rm(AP)} \cite[Theorem~C]{LdlS11} does not decide ordinary {\rm ITAP}. It yields an operator-space fixed-point defect. By \cite[Corollary~1.8]{KU14}, there are a Hilbert space \(H\) and an operator space \(S\subset\mathcal K(H)\) such that, for the right adjoint action on the first tensor factor and the trivial action on \(S\),
	\[
	\bigl(\rC_u^*(\SL_3(\Z))\otimes_{\min} S\bigr)^{\SL_3(\Z)}
	\neq
	\rC_u^*(\SL_3(\Z))^{\SL_3(\Z)}\otimes_{\min} S.
	\]
	This coefficient-valued defect does not by itself imply failure of the scalar
	equality defining ordinary {\rm ITAP}. Thus failure of {\rm(AP)} leaves the scalar
	upper-bound case open, and Theorem~\ref{thm:semisimple-itap-obstruction} identifies its
	exact remaining issue as ordinary {\rm ITAP}.
	
	The implication from {\rm(AP)} to {\rm ITAP} is due to Zacharias
	\cite[Theorem~3.2]{Za06}. For completeness, we give a short proof using Suzuki's
	Fourier reconstruction theorem. For the left-translation action identify
	\(\rC_u^*(\Gamma)=\ell^\infty(\Gamma)\rtimes_r\Gamma\), and take
	\(T\in\rC_u^*(\Gamma)\cap\rL(\Gamma)\). For \(s,r\in\Gamma\), one has
	\[
	\rE(Tu_s^*)(r)=\langle T\delta_{s^{-1}r},\delta_r\rangle.
	\]
	Since
	\(T\in\rL(\Gamma)\), the right-hand side depends only on
	\(r(s^{-1}r)^{-1}=s\), and hence every Fourier coefficient of \(T\) is scalar.
	Proposition~\ref{prop:suzuki}, applied to
	\(\C1\subset\ell^\infty(\Gamma)\), places \(T\) in
	\(\C1\rtimes_r\Gamma=\rC^*_{\lambda}(\Gamma)\). The reverse inclusion is
	automatic. In particular, this applies to the rank-one product lattices of
	Theorem~\ref{thm:classification}.

	% Bibliography labels use publication years.

\end{document}